\documentclass[oneside,english]{amsart}
\usepackage[T1]{fontenc}
\usepackage[latin9]{inputenc}
\usepackage{amsthm}
\usepackage{amssymb}
\PassOptionsToPackage{normalem}{ulem}
\usepackage{ulem}

\makeatletter
\numberwithin{equation}{section}
\numberwithin{figure}{section}
\theoremstyle{plain}
\newtheorem{thm}{\protect\theoremname}[section]
  \theoremstyle{definition}
  \newtheorem{defn}[thm]{\protect\definitionname}
  \theoremstyle{remark}
  \newtheorem{rem}[thm]{\protect\remarkname}
  \theoremstyle{plain}
  \newtheorem*{question*}{\protect\questionname}
\newcounter{mainthm}
\theoremstyle{plain}

\newtheorem{main_thm}[mainthm]{Main Theorem}
  \theoremstyle{plain}
  \newtheorem{cor}[thm]{\protect\corollaryname}
  \theoremstyle{plain}
  \newtheorem{question}[thm]{\protect\questionname}
  \theoremstyle{remark}
  \newtheorem{claim}[thm]{\protect\claimname}
  \theoremstyle{plain}
  \newtheorem{prop}[thm]{\protect\propositionname}
  \theoremstyle{plain}
  \newtheorem{fact}[thm]{\protect\factname}
  \theoremstyle{plain}
  \newtheorem{lem}[thm]{\protect\lemmaname}

\usepackage{amsmath}

\usepackage{a4wide}
\usepackage{amssymb}
\usepackage{url}
\makeatother

\usepackage{babel}
  \providecommand{\claimname}{Claim}
  \providecommand{\corollaryname}{Corollary}
  \providecommand{\definitionname}{Definition}
  \providecommand{\factname}{Fact}
  \providecommand{\lemmaname}{Lemma}
  \providecommand{\propositionname}{Proposition}
  \providecommand{\questionname}{Question}
  \providecommand{\remarkname}{Remark}
\providecommand{\theoremname}{Theorem}

\begin{document}
\global\long\def\C{\mathfrak{C}}
\global\long\def\P{\mathbf{P}}
\global\long\def\p{\mathbf{p}}
\global\long\def\q{\mathbf{q}}
\global\long\def\SS{\mathcal{P}}
\global\long\def\Zz{\mathbb{Z}}
 \global\long\def\mod{\mbox{mod}}
\global\long\def\pr{\mbox{pr}}
\global\long\def\image{\mbox{im}}
\global\long\def\otp{\mbox{otp}}
\global\long\def\dec{\mbox{dec}}
\global\long\def\ist{\mbox{ist}}
\global\long\def\pre{\mbox{\dpr}}
 \global\long\def\lev{\mbox{\mbox{lev}}}
\global\long\def\Suc{\mbox{\mbox{Suc}}}
\global\long\def\Aut{\mbox{Aut}}
\global\long\def\minb{{\min}}
\global\long\def\concat{\frown}
\global\long\global\long\global\long\def\alg{\operatorname{alg}}
 \global\long\global\long\global\long\def\sep{\operatorname{sep}}
\global\long\def\NTPT{\operatorname{NTP}_{\operatorname{2}}}
\global\long\def\tp{\operatorname{tp}}
\global\long\def\dpr{\operatorname{rk-dp}}
\global\long\def\id{\operatorname{id}}

\def\Ind#1#2{#1\setbox0=\hbox{$#1x$}\kern\wd0\hbox to 0pt{\hss$#1\mid$\hss} \lower.9\ht0\hbox to 0pt{\hss$#1\smile$\hss}\kern\wd0} 
\def\Notind#1#2{#1\setbox0=\hbox{$#1x$}\kern\wd0\hbox to 0pt{\mathchardef \nn="3236\hss$#1\nn$\kern1.4\wd0\hss}\hbox to 0pt{\hss$#1\mid$\hss}\lower.9\ht0 \hbox to 0pt{\hss$#1\smile$\hss}\kern\wd0} 

  \theoremstyle{definition}
  \newtheorem{defClaim}[thm]{Definition/Claim}

\begin{flushleft}
\global\long\def\ind{\mathop{\mathpalette\Ind{}}}
 \global\long\def\nind{\mathop{\mathpalette\Notind{}}}

\par\end{flushleft}

\global\long\def\alt{\mbox{\mbox{alt}}}
\global\long\def\seq{\mbox{\mbox{seq}}}

\title{Witnessing dp-rank}

\author{Itay Kaplan and Pierre Simon}
\begin{abstract}
We prove that in $\NTPT$ theories if $p$ is a dependent type with
dp-rank $\geq\kappa$, then this can be witnessed by indiscernible
sequences of tuples satisfying $p$. If $p$ has dp-rank infinity,
then this can be witnessed by singletons (in any theory). 
\end{abstract}

\thanks{The first author's research was partially supported by the SFB 878
grant. }

\maketitle

\section{Introduction}

In this note we answer a question of Alf Onshuus and Alexander Usvyatsov,
whether dp-minimality can be witnessed by indiscernible sequences
of singletons. We prove two general theorems regarding dp-rank.

\medskip{}

Let $Card$ denote the class of cardinals. We define $Card^{*}$ to
be the class $Card$ to which we add an element $\kappa_{-}$ for
each infinite cardinal $\kappa.$ We extend the linear order from
$Card$ to $Card^{*}$ by setting $\mu<\kappa_{-}<\kappa$ whenever
$\mu<\kappa$ are cardinals.
\begin{defn}
\label{def:dp-rank} Let $p\left(x\right)$ be a partial (consistent)
type over a set $A$ ($x$ is a finite tuple, here and throughout
the paper). We define the \emph{dp-rank} of $p\left(x\right)$ (which
is an element of $Card^{*}$ or $\infty$) as follows:
\begin{itemize}
\item Let $\kappa$ be a cardinal. We will say that $p\left(x\right)$ has
dp-rank $<\kappa$ (which we write $\dpr\left(p\right)<\kappa$) if
given any realization $a$ of $p$ and any $\kappa$ mutually indiscernible
sequences over $A$, at least one of them is indiscernible over $Aa$.
\item We say that $p$ has dp-rank $\kappa$ over $A$ (or $\dpr\left(p\right)=\kappa$)
if it has dp-rank $<\mu$ for all $\mu>\kappa$, but it is not the
case that $\dpr\left(p\right)<\kappa$.
\item If $\kappa$ is an infinite cardinal, we say that $p$ has dp-rank
$\kappa_{-}$ over $A$ (or $\dpr\left(p\right)=\kappa_{-}$) if it
has dp-rank $<\kappa$, but for no $\mu<\kappa$ do we have $\dpr\left(p\right)<\mu$.
\item If $\dpr\left(p\right)<\kappa$ holds for no cardinal $\kappa$, then
we say that $p$ has dp-rank $\infty$.
\item We call $p$ \emph{dp-minimal }if it has dp-rank $1$.
\item We call $p$ \emph{dependent} if $\dpr\left(p\right)<\infty$. This
is equivalent to $\dpr\left(p\right)<\left|T\right|^{+}$ (see Corollary
\ref{cor:infinite is even more bounded}).
\end{itemize}
\end{defn}
\begin{rem}
\label{rem:Base Set}It is easy to see that the set $A$ does not
matter, as long as $p$ is defined over it. Indeed, for a set $B$
over which $p$ is defined, let us define for the sake of discussion
$\dpr\left(p,B\right)$ as the dp-rank of $p$ over $B$ similarly
to the definition above but we add the requirement that the sequences
are mutually indiscernible over $B$. If $A\subseteq B$, and $p$
is a type over $A$, then it is easy to see that $\dpr\left(p,B\right)\leq\dpr\left(p,A\right)$
while the other direction uses a standard application of Ramsey theorem,
so $\dpr\left(p,B\right)=\dpr\left(p,A\right)$. 
\end{rem}
Note also that if $q\left(x\right)$ extends $p\left(x\right)$ then
$\dpr\left(p\left(x\right)\right)\geq\dpr\left(q\left(x\right)\right)$,
so:
\begin{rem}
\label{prop:extension of dependent}Any extension of a dependent type
is dependent.
\end{rem}
Recall:
\begin{defn}
A (complete, first order) theory $T$ is \emph{dp-minimal} if the
type $\left\{ x=x\right\} $ is dp-minimal. The theory $T$ is \emph{dependent}
if the type $\left\{ x=x\right\} $ is dependent.
\end{defn}
Dp-rank and dependent types were originally defined in \cite{Us}
and further studied in \cite{OnUs1}. Dp-rank is a simplification
of the various ranks appearing in \cite{Sh863}. We use a slightly
different convention for it than those two papers which has the advantage
of distinguishing between $\kappa$ and $\kappa_{-}$. Yet another
convention is used in \cite{AlexAlfTemp} which has the disadvantage
of giving a different meaning to $\dpr\left(p\right)=\kappa$ depending
on whether $\kappa$ is finite or infinite. Dp-minimality was first
defined in \cite{OnUs1}. It is shown in \cite{Simon-Dp-min} that
the original definition of dp-minimality is equivalent to the definition
given here.

Examples of dp-minimal theories include all o-minimal theories and
C-minimal theories. 

Note that the sequences that witness $\dpr\left(p\right)\geq\kappa$
in Definition \ref{def:dp-rank} can always be taken to be sequences
of finite tuples, but can we bound the length?
\begin{question*}
\label{que:OU}(A. Onshuus, A. Usvyatsov) Can we assume in the definition
of dp-minimality that the indiscernible sequences are sequences of
singletons?
\end{question*}
We provide a positive answer in Corollary \ref{cor:Answer} below,
but we need to add parameters to the base. 

We prove the following two theorems:
\begin{main_thm}
\label{MainThmInDependent}If $p$ is a type over $A$ which is independent
(i.e. $\dpr\left(p\right)=\infty$), then there is some $A'\supseteq A$
such that $\left|A'\backslash A\right|$ is finite, a realization
$a\models p$ and $A'$-mutually indiscernible sequences of \uline{singletons}
$\left\langle I_{i}\left|\, i<\left|T\right|^{+}+\left|A\right|^{+}\right.\right\rangle $
such that $I_{i}$ is not indiscernible over $A'a$ for all $i$. 
\end{main_thm}
From this we will deduce:
\begin{cor}
\label{cor:To-check-whether}To check whether a theory is dependent
it is enough to check that for every indiscernible sequence of singletons
$\left\langle a_{i}\left|\, i<\left|T\right|^{+}\right.\right\rangle $
over some finite $A$, and for every singleton $c$, there is $\alpha<\left|T\right|^{+}$
such that $\left\langle a_{i}\left|\, i>\alpha\right.\right\rangle $
is indiscernible over $Ac$. 
\end{cor}
The second result is about dependent types, but to prove it we need
to assume%
\footnote{After the appearance of this note, Artem Chernikov has removed this
assumption, see \cite{ArtemNTP2}.%
} that the theory is $\NTPT$.
\begin{defn}
A theory $T$ is\emph{ $\NTPT$} \emph{(does not have the tree property
of the second kind)} if there is no formula $\varphi\left(x,y\right)$
and array $\left\langle a_{i,j}\left|\, i,j<\omega\right.\right\rangle $
such that for every $i<\omega$, $\left\{ \varphi\left(x,a_{i,j}\right)\left|\, j<\omega\right.\right\} $
is $k$-inconsistent (i.e. each subset of size $k$ is inconsistent)
and for every $\eta:\omega\to\omega$, the set $\left\{ \varphi\left(x,a_{i,\eta\left(i\right)}\right)\left|\, i<\omega\right.\right\} $
is consistent.
\end{defn}
The class of $\NTPT$ theories contains both simple and dependent
theories. 
\begin{main_thm}
\label{MainThmDep}Assume $T$ is $\NTPT$, and that $p$ is a dependent
type over $A$ with $\dpr\left(p\right)\geq\kappa$. Then there is
some $A'\supseteq A$, some $a\models p$ and $A'$-mutually indiscernible
sequences $\left\{ I_{i}\left|\, i<\kappa\right.\right\} $ such that
each of them is not indiscernible over $A'a$ \uline{and} all tuples
in each $I_{i}$ satisfy $p$. 
\end{main_thm}
Note that we may always choose $A'$ so that $|A'\setminus A|$ is
at most $\kappa+\aleph_{0}$ since, for each sequence $I_{i}$, we
only need finitely many parameters from $A'$ to witness that $I_{i}$
is not indiscernible over $A'a$.

Now we can answer Question \ref{que:OU}:
\begin{cor}
\label{cor:Answer}If $T$ is not dp-minimal, then there is some finite
set $A'$, some singleton $a$ and two $A'$-mutually indiscernible
sequences $\left\{ I,J\right\} $ \uline{of singletons} such that
both $I$ and $J$ are not indiscernible over $A'a$.\end{cor}
\begin{proof}
Right to left is obvious. For the other direction, if $T$ is dependent
then we may use Main Theorem \ref{MainThmDep} (since there are only
two sequences, only finitely many parameters from $A'$ are needed
to witness non-indiscernibility, so we may assume that $A'$ is finite).
But if $T$ is not dependent, then by Main Theorem \ref{MainThmInDependent}
there exists such $a$, $A$ and infinitely many such sequences. 
\end{proof}
The following question remains open:
\begin{question}
(J. Ramakrishnan) Can we assume in the definition of dp-rank that
the indiscernible sequences are sequences of singletons by adding
parameters to the base?
\end{question}
Our results show that this is indeed the case when the type is independent
or when it is the type of a singleton in an $\NTPT$ theory. 

In Section \ref{sec:Proof-of-InDep} we prove Main Theorem \ref{MainThmInDependent},
and in Section \ref{sec:Proof-of-Dep} we prove Main Theorem \ref{MainThmDep}.
\begin{question}
Are the extra parameters in the Main Theorems needed?
\end{question}
Throughout the paper, $\C$ will denote a monster model of the theory
$T$ (i.e. a very big saturated model).

\section{\label{sec:Proof-of-InDep}On dependent types and a proof of Main
Theorem \ref{MainThmInDependent}}

\subsection{On dependent types}

We start with the following easy observation (which is somewhat similar
to \cite[Observation 2.7]{OnUs1}), with a very straightforward proof. 
\begin{claim}
\label{cla:infinite is bounded}Suppose $p\left(x\right)$ is a partial
type over $A$. Then the following are equivalent:
\begin{enumerate}
\item There is $a\models p$ and $A$-mutually indiscernible sequences $\left\langle I_{i}\left|\, i<\omega\right.\right\rangle $
such that the sequence $\left\langle I_{i}\left|\, i<\omega\right.\right\rangle $
is indiscernible over $Aa$, and for each $i$, $I_{i}$ is not indiscernible
over $Aa$. 
\item $p$ is independent.
\item $\dpr\left(p\right)\geq\left|T\right|^{+}+\left|A\right|^{+}$.
\item There is an $A$-indiscernible sequence $\left\langle a_{i}\left|\, i<\omega\right.\right\rangle $
such that $a_{i}\models p$, a formula $\varphi\left(x,y\right)$
and some $c$ such that $\varphi\left(a_{i},c\right)$ holds iff $i$
is even. 
\item There is an $A$-indiscernible sequence $\left\langle b_{i}\left|\, i<\omega\right.\right\rangle $,
a formula $\psi\left(y,x\right)$ and some $d\models p$ such that
$\psi\left(b_{i},d\right)$ holds iff $i$ is even.
\item There is a set $\left\{ a_{i}\left|\, i<\omega\right.\right\} $ of
realizations of $p$ and a formula $\varphi\left(x,y\right)$ such
that for every $s\subseteq\omega$, there is some $c_{s}$ such that
$\varphi\left(a_{i},c_{s}\right)$ holds iff $i\in s$.
\item There is a set $\left\{ b_{i}\left|\, i<\omega\right.\right\} $ and
a formula $\psi\left(y,x\right)$ such that for every $s\subseteq\omega$,
there is some $d_{s}\models p$ such that $\psi\left(b_{i},d_{s}\right)$
holds iff $i\in s$.
\end{enumerate}
\end{claim}
\begin{proof}
(1) implies (2) and (2) implies (3) are easy. Assume (3) and show
(1). We can find $a\models p$ and $A$-mutually indiscernible sequences
$\left\langle I_{i}\left|\, i<\left|T\right|^{+}+\left|A\right|^{+}\right.\right\rangle $
such that for all $i$, $I_{i}$ is not indiscernible over $Aa$.
We may assume that the order type of these sequences is $\omega$.
The fact that $I_{i}$ is not indiscernible over $Aa$ is witnessed
by some formula over $A$ and increasing tuples from $I_{i}$, so
we may assume that for infinitely many $i$, the formula is the same,
and the position of these tuples does not depend on $i$ (maybe changing
$a$). Then, by Ramsey and compactness, we may assume that $\left\langle I_{i}\left|\, i<\omega\right.\right\rangle $
is indiscernible over $Aa$. 

(5) follows from (1): Denote $I_{i}=\left\langle a_{i,j}\left|\, j<\omega\right.\right\rangle $.
There is a formula $\psi\left(x,y\right)$ over $A$ and an increasing
tuple $k_{0}<\ldots<k_{n-1}<r_{0}<\ldots<r_{n-1}$ such that, letting
$a_{i,\bar{k}}=\left(a_{i,k_{0}},\ldots,a_{i,k_{n-1}}\right)$ (and
similarly we define $a_{i,\bar{r}}$), $\psi\left(a_{i,\bar{k}},a\right)\land\neg\psi\left(a_{i,\bar{r}},a\right)$
holds for all $i<\omega$. The sequence $\left\langle b_{i}\left|\, i<\omega\right.\right\rangle $
defined by $b_{i}=a_{i,\bar{k}}$ when $i$ is even and $b_{i}=b_{i,\bar{r}}$
when $i$ is odd satisfies (5). The fact that $\psi$ is over $A$
is no problem --- we can add the parameters to $b_{i}$.

(2) follows from (5) is easy by compactness.

(6) is equivalent to (4) and (7) is equivalent to (5) by a standard
application of Ramsey. 

(6) follows from (5): By indiscernibility, we may extend $\left\langle b_{i}\left|\, i<\omega\right.\right\rangle $
to $\left\langle b_{r}\left|\, r\in\SS\left(\omega\right)\right.\right\rangle $
(with some ordering), and so, for every subset $s\subseteq\SS\left(\omega\right)$,
there is some $d_{s}\models p$ such that $\psi\left(b_{r},d_{s}\right)$
iff $r\in s.$ For $i<\omega$, let $d_{i}=d_{\left\{ r\subseteq\omega:\, i\in r\right\} }$.
Then for each subset $r\subseteq\omega$, $\psi\left(b_{r},d_{i}\right)$
iff $i\in r$. This gives us (6). The same exact argument gives that
(7) follows from (4). \end{proof}
\begin{prop}
\label{prop:local char} If $p$ is a dependent type over $A$, then
there is $B\subseteq A$ of size $\left|B\right|\leq\left|T\right|$
such that $p|_{B}$ is dependent. \end{prop}
\begin{proof}
By Claim \ref{cla:infinite is bounded} (6), it cannot be that there
exists a formula $\varphi\left(x,y\right)$ and a set $\left\{ a_{i}\left|\, i<\omega\right.\right\} $
of realizations of $p$ such that for each $s\subseteq\omega$, there
is some $c_{s}$ such that $\varphi\left(a_{i},c_{s}\right)$ holds
iff $i\in s$. By compactness, there is no formula $\varphi\left(x,y\right)$
such that for all finite $B\subseteq A$ we can find such a set $\left\{ a_{i}\left|\, i<\omega\right.\right\} $
of realizations of $p|_{B}$ and such $c_{s}$ for $s\subseteq\omega$.
So for each formula $\varphi\left(x,y\right)$, there is some finite
$B_{\varphi}\subseteq A$ such that there is no such set. Let $B=\bigcup_{\varphi}B_{\varphi}$.
Then $p|_{B}$ is easily seen to be dependent. \end{proof}
\begin{cor}
\label{cor:infinite is even more bounded}The following are equivalent
for a type $p\left(x\right)$ over $A$:
\begin{enumerate}
\item $p\left(x\right)$ is independent.
\item \textup{$\dpr\left(p\right)\geq\left|T\right|^{+}$.}
\end{enumerate}
\end{cor}
\begin{proof}
If $p$ is dependent, then there is some $B\subseteq A$ such that
$p|_{B}$ is dependent and $\left|B\right|\leq\left|T\right|$. By
Claim \ref{cla:infinite is bounded} (3), this means that $\dpr\left(p|_{B}\right)<\left|T\right|^{+}$,
so $\dpr\left(p\right)<\left|T\right|^{+}$. 
\end{proof}
In this section we show that some useful properties that are true
in dependent theories are actually true in the local context as well. 
\begin{fact}
\label{fac:2variablesDep}\cite[Theorem 4.11]{AlexAlfTemp} If $p$
is a dependent type over $A$, and $a_{i}\models p$ for $i<n<\omega$,
then $\tp\left(a_{0},\ldots,a_{n-1}/A\right)$ is also dependent.
\end{fact}
Recall the notions of forking and dividing. All the definitions and
properties we need can be found in \cite{Kachernikov}. 
\begin{prop}
\label{prop:non-forking}If $p$ is dependent type over a model $M$,
and $q$ is a global non-forking extension of $p$ (i.e. an extension
to $\C$), then $q$ is invariant over $M$.\end{prop}
\begin{proof}
Suppose that $\varphi\left(x,c_{0}\right)\land\neg\varphi\left(x,c_{1}\right)\in q$
where $c_{0}\equiv_{M}c_{1}$. Then using a standard technique, we
can assume that $c_{0},c_{1}$ start an indiscernible sequence $\left\langle c_{0},c_{1},\ldots\right\rangle $
over $M$. The set 
\[
p\left(x\right)\cup\left\{ \varphi\left(x,c_{i}\right)^{\left(i\mbox{ is even}\right)}\left|\, i<\omega\right.\right\} 
\]
 is inconsistent by Claim \ref{cla:infinite is bounded}. This means
that for some formula $\psi\left(x\right)\in p$, 
\[
\left\{ \psi\left(x\right)\wedge\varphi\left(x,c_{2i}\right)\wedge\neg\varphi\left(x,c_{2i+1}\right)\left|\, i<\omega\right.\right\} 
\]
 is inconsistent, and so $\psi\left(x\right)\wedge\varphi\left(x,c_{0}\right)\wedge\neg\varphi\left(x,c_{1}\right)$
divides over $M$ --- contradiction. \end{proof}
\begin{prop}
\label{prop:(shrinking-of-indiscernibles)}(shrinking of indiscernibles)
Suppose that $p\left(x\right)$ is a dependent type over $A$ and
that $B$ is a set of realizations of $p$. 

If $I=\left\langle a_{i}\left|\, i<\left|T\right|^{+}+\left|B\right|^{+}\right.\right\rangle $
is an $A$-indiscernible sequence, then some end-segment is indiscernible
over $AB$. Note that the size of $A$ and the size of the tuple $a_{i}$
do not matter.\end{prop}
\begin{proof}
We may assume that $B$ is finite. The type $\tp\left(B/A\right)$
is dependent by Fact \ref{fac:2variablesDep}. The proof easily follows
from Corollary \ref{cor:infinite is even more bounded}. 
\end{proof}

\subsection{Proof of Main Theorem \ref{MainThmInDependent}}
\begin{defn}
Let $p\left(x\right)$ be a type over $A$. We say that $p$ is\emph{
$1$-independent over $A$ }if there is a realization $a\models p$
and $A$-mutually indiscernible sequences $\left\langle I_{i}\left|\, i<\omega\right.\right\rangle $
of \uline{singletons} such that the sequence $\left\langle I_{i}\left|\, i<\omega\right.\right\rangle $
is indiscernible over $Aa$ and for each $i<\omega$, $I_{i}$ is
not indiscernible over $Aa$. 

We say that $p$ is \emph{$1$-dependent} \emph{over $A$} if it is
not $1$-independent over $A$. We say that $p$ is \emph{$1$-dependent}
if it is $1$-dependent over any $A'\supseteq A$ such that $A'\backslash A$
is finite.
\end{defn}
Observe that by Claim \ref{cla:infinite is bounded}, if $p\left(x\right)$
is dependent then it is $1$-dependent. Also, as in Remark \ref{rem:Base Set},
this definition does not depend on $A$. 
\begin{claim}
\label{cla:1-dep-shrinking}If $p\left(x\right)$ is a type over $A$
which is $1$-dependent, then:
\begin{itemize}
\item For every $A'\supseteq A$ such that $A'\backslash A$ is finite,
every $A'$-indiscernible sequence 
\[
\left\langle a_{i}\left|\, i<\left|T\right|^{+}+\left|A\right|^{+}\right.\right\rangle 
\]
 of tuples satisfying $p$ and singleton $c$, there is some $\alpha<\left|T\right|^{+}+\left|A\right|^{+}$
such that the end-segment $\left\langle a_{i}\left|\,\alpha<i\right.\right\rangle $
is indiscernible over $A'c$.
\end{itemize}
\end{claim}
\begin{proof}
To simplify notations, assume $A=A'=\emptyset$. Towards a contradiction
we find a formula $\varphi\left(\bar{x},y\right)$ and an indiscernible
sequence $\left\langle \bar{a}_{i}\left|\, i<\omega\right.\right\rangle $
such that $\bar{a}_{i}$ is a tuple of length $n$ of tuples satisfying
$p$ and $\varphi\left(\bar{a}_{i},c\right)$ holds iff $i$ is even.
By the proof of Claim \ref{cla:infinite is bounded} (i.e. (5) implies
(4), with $p=\tp\left(c\right)$), there is an indiscernible sequence
$\left\langle c_{\bar{i}}\left|\,\bar{i}\in\omega^{n+1}\right.\right\rangle $
(ordered lexicographically) of singletons such that $\varphi\left(\bar{a}_{0},c_{\bar{i}}\right)$
holds iff the last number in $\bar{i}$ is even. We may also assume
(by Ramsey) that the sequence $\left\langle \bar{c}_{\bar{i}}\left|\,\bar{i}\in\omega^{n}\right.\right\rangle $
is indiscernible over $\bar{a}_{0}$, where $\bar{c}_{\bar{i}}=\left\langle c_{\bar{i}\concat j}\left|\, j<\omega\right.\right\rangle $. 

Suppose $\bar{a}_{0}=\left(a_{0,0},\ldots,a_{0,n-1}\right)$ where
$a_{0,i}\models p$. Since $p$ is $1$-dependent over $\emptyset$,
there is some $i_{0}<\omega$ such that $\left\langle c_{i_{0}\concat\bar{i}}\left|\,\bar{i}\in\omega^{n}\right.\right\rangle $
is indiscernible over $a_{0,0}$. By assumption, $p$ is $1$-dependent
over $a_{0,0}$. Inductively, we can find $i_{1},\ldots,i_{n-1}<\omega$
such that $\bar{c}_{\left(i_{0},\ldots,i_{n-1}\right)}$ is indiscernible
over $\bar{a}_{0}$ --- contradiction. 
\end{proof}
The following theorem implies Main Theorem \ref{MainThmInDependent}:
\begin{thm}
\label{thm:MainA}If $p\left(x\right)$ is a type over $A$ which
satisfies the conclusion of Claim \ref{cla:1-dep-shrinking}, then
it is dependent.\end{thm}
\begin{proof}
Again, assume $A=\emptyset$. Suppose $p$ is a counterexample. By
Claim \ref{cla:infinite is bounded}, there is an indiscernible sequence
$\left\langle a_{i}\left|\, i<\left|T\right|^{+}\right.\right\rangle $
such that $a_{i}\models p$, a formula $\varphi\left(x,y\right)$
and some tuple $c=\left(c_{0},\ldots,c_{n-1}\right)$ such that $\varphi\left(a_{i},c\right)$
holds iff $i$ is even. By assumption, there is some end-segment which
is indiscernible over $c_{0}$. Applying the conclusion of Claim \ref{cla:1-dep-shrinking}
again with $A'=\left\{ c_{0}\right\} $, we get an end-segment which
is indiscernible over $c_{0}c_{1}$. Continuing like this, we get
an end-segment which is indiscernible over $c$ --- contradiction. 
\end{proof}
Since dependent implies $1$-dependent, we get:
\begin{cor}
\label{cor:1-dep-iff}The type $p\left(x\right)$ is $1$-dependent
iff it is dependent iff it satisfies the conclusion of Claim \ref{cla:1-dep-shrinking}. 
\end{cor}
Corollary \ref{cor:To-check-whether} follows:
\begin{cor}
A theory $T$ is dependent iff for every indiscernible sequence of
singletons 
\[
\left\langle a_{i}\left|\, i<\left|T\right|^{+}\right.\right\rangle 
\]
 over some finite $A$, and for every singleton $c$, there is $\alpha<\left|T\right|^{+}$
such that $\left\langle a_{i}\left|\,\alpha<i\right.\right\rangle $
is indiscernible over $Ac$. \end{cor}
\begin{proof}
Apply Corollary \ref{cor:1-dep-iff} with $p\left(x\right)=\left\{ x=x\right\} $. 
\end{proof}

\section{\label{sec:Proof-of-Dep}Proof of Main Theorem \ref{MainThmDep}}

\subsection{Preliminaries on $\NTPT$ theories}

From here up to the end of the section, we assume that the theory
is $\NTPT$.

In the study of forking in $\NTPT$ theories, it is sometimes useful
to consider independence relations. For instance, we denote $a\ind_{B}^{f}C$
for $\tp\left(a/BC\right)$ does not fork over $B$. Similarly, $a\ind_{B}^{i}C$
means that there is a global extension (i.e. an extension to $\C$)
of $\tp\left(a/BC\right)$ which is Lascar invariant over $B$, meaning
that if $d$ and $c$ have the same Lascar strong type over $B$ then
either both $\varphi\left(x,c\right)$ and $\varphi\left(x,d\right)$
are in this extension or neither of them is. We do not really need
Lascar strong type in this section, because we only work over models.
Over a model, Lascar invariance is the same as invariance.

In the proofs we shall only use the following facts about $\NTPT$
theories. These were proved in \cite{Kachernikov}. 
\begin{defn}
\label{def:strictInv} (strict invariance) We say that $\tp\left(a/Bb\right)$
is strictly invariant over $B$ (denoted by $a\ind_{B}^{\ist}b$)
if there is a global extension $p$, which is Lascar invariant over
$B$ (so $a\ind_{B}^{i}b$) and for any $C\supseteq Bb$, if $c\models p|_{C}$
then $C\ind_{B}^{f}c$. \end{defn}
\begin{fact}
\label{fac:forking in NTPT}In $\NTPT$ theories
\begin{enumerate}
\item Forking equals dividing over models.
\item ``Kim's Lemma'': If $\varphi\left(x,a\right)$ divides over $A$,
and $\left\langle b_{i}\left|i<\omega\right.\right\rangle $ is a
sequence satisfying $b_{i}\equiv_{A}a$ and $b_{i}\ind_{A}^{\ist}b_{<i}$,
then $\left\{ \varphi\left(x,b_{i}\right)\left|\, i<\omega\right.\right\} $
is inconsistent. In particular, if $\left\langle b_{i}\left|i<\omega\right.\right\rangle $
is an indiscernible sequence then it witnesses dividing of $\varphi\left(x,a\right)$. 
\end{enumerate}
\end{fact}
Recall:
\begin{defn}
Suppose $p$ is a global type which is invariant over a set $A$. 
\begin{enumerate}
\item We say that a sequence $\left\langle a_{i}\left|\, i<\alpha\right.\right\rangle $
is a Morley sequence of a type $p$ over $B\supseteq A$ if $a_{0}\models p|_{B}$
and for all $i<\alpha$, $a_{i}\models p|_{Ba_{<i}}$. This is an
indiscernible sequence over $B$.
\item We let the type $p^{\left(\alpha\right)}$ be the union of $\tp\left(\left\langle a_{i}\left|\, i<\alpha\right.\right\rangle /B\right)$
running over all $B\supseteq A$. This is again an $A$-invariant
type.
\item If $q$ is also an $A$-invariant global type, we define $p\otimes q$
as the union of $\tp\left(a,b/B\right)$ running over all $B\supseteq A$
where $a\models p|_{B}$ and $b\models q|_{Ba}$. This is also an
$A$-invariant global type.
\item Similarly, given a sequence $\left\langle p_{i}\left|\, i<\alpha\right.\right\rangle $
of $A$-invariant global types, we define $\bigotimes_{i<\alpha}p_{i}$
as the union of $\tp\left(\left\langle a_{i}\left|\, i<\alpha\right.\right\rangle /B\right)$
running over all $B\supseteq A$, where $a_{i}\models p_{i}|_{Ba_{<i}}$.
Again, this is an $A$-invariant global type.
\end{enumerate}
\end{defn}
In the definition above, all types may have infinitely many variables.
\begin{rem}
\label{rem:Morley Seq mut. indis}If $\left\{ J_{0},\ldots,J_{k}\right\} $
is a set of mutually indiscernible sequences over $C\supseteq A$,
and $\left\langle a_{i}\left|\, i<\alpha\right.\right\rangle $ is
a Morley sequence of a global $A$-invariant type over $\left\{ J_{0},\ldots,J_{k}\right\} \cup C$
then $\left\{ J_{0},\ldots,J_{k},\left\langle a_{i}\left|\, i<\alpha\right.\right\rangle \right\} $
is mutually indiscernible over $C$.

(Why? On the one hand, $\left\{ J_{0},\ldots,J_{k}\right\} $ is mutually
indiscernible over $C\cup\left\{ a_{i}|\, i<\alpha\right\} $ since
$\tp\left(\left\langle a_{i}\left|\, i<\alpha\right.\right\rangle /\left\{ J_{0},\ldots,J_{k}\right\} \cup C\right)$
does not split over $A$. On the other hand, $\langle a_{i}|\, i<\alpha\rangle$
is a Morley sequence over $\left\{ J_{0},\ldots,J_{k}\right\} \cup C$,
and as such is indiscernible over that set.)
\end{rem}
We also need to recall the notions of heir and coheir:
\begin{defn}
A global type $p\left(x\right)$ is called a coheir over a set $A$,
if it is finitely satisfiable in $A$. Note that in this case, it
is invariant over $A$, and $p^{\left(\alpha\right)}$ is also a coheir
over $A$. 

It is called an heir over $A$ if for every formula over $A$, $\varphi\left(x,b\right)\in p$,
there exists some $a'\in A$ such that $\varphi\left(x,a'\right)\in p$.\end{defn}
\begin{claim}
\label{cla:heir-coheir}If $p$ is an $A$-invariant global type and
$p^{\left(\omega\right)}$ is both an heir and a coheir over $A$,
then any Morley sequence of $p$ over $A$, $\left\langle a_{i}\left|\, i<\omega\right.\right\rangle $
satisfies $a_{\geq i}\ind_{A}^{\ist}a_{<i}$ for any $i<\omega$.\end{claim}
\begin{proof}
The type $p^{\left(\omega\right)}$ is a global $A$-invariant (so
also $A$-Lascar invariant) type that extends $\tp\left(a_{\geq i}/Aa_{<i}\right)$,
and if $c\models p^{\left(\omega\right)}|AC$ then $\tp\left(C/Ac\right)$
is finitely satisfiable over $A$ (since $p^{\left(\omega\right)}$
is an heir over $A$), and it follows that $C\ind_{A}^{f}c$. \end{proof}
\begin{claim}
\label{cla:Making heir}Given any global type $p\left(x\right)$ and
a set $A$, we can find a model $M\supseteq A$ such that $p$ is
an heir over $M$. \end{claim}
\begin{proof}
Construct inductively a sequence of models $M_{i}$ for $i<\omega$.
Let $M_{0}$ be any model containing $A$. Let $M_{i+1}\supseteq M_{i}$
be such that for every formula $\varphi\left(x,y\right)$ over $M_{i}$,
if $\varphi\left(x,a\right)\in p$ then there is some such $a$ in
$M_{i+1}$. Finally, let $M=\bigcup_{i<\omega}M_{i}$. \end{proof}
\begin{lem}
\label{lem:IndImpliesNonForking}Let $M$ be a model. Suppose that
$p$ is an $M$-invariant global type such that $p^{\left(\omega\right)}$
is an heir-coheir over $M$. Suppose $I$ is an endless Morley sequence
of $p$ over $M$. If $I$ is indiscernible over $Ma$ then $\tp\left(a/MI\right)$
does not fork over $M$.\end{lem}
\begin{proof}
By Fact \ref{fac:forking in NTPT}, it is enough to see that the type
does not divide over $M$. Suppose $\varphi\left(x,b_{0}\right)\in\tp\left(a/MI\right)$
divides over $M$, where $\varphi$ is over $M$ and $b_{0}\subseteq I$.
For $i\geq1$ choose tuples $b_{i}\subseteq I$ of the same length
as $b_{0}$ that appear after $b_{0}$ in increasing order. By Claim
\ref{cla:heir-coheir}, $b_{i}\ind_{M}^{\ist}b_{<i}$ so by ``Kim's
lemma'' (Fact \ref{fac:forking in NTPT}), it must witness dividing.
But this is a contradiction to the fact that $I$ is indiscernible
over $Ma$.
\end{proof}

\subsection{Proof of the main theorem}

The following is the key definition in the proof.
\begin{defn}
\label{def:altType}Suppose
\begin{enumerate}
\item $p$ is a global $A$-invariant type such that $p|_{A}$ is dependent
(we call such types $A$-invariant and $A$-dependent).
\item $B$ is some set containing $A$. 
\item $\varphi\left(x,y\right)$ is a formula over $A$.
\item $a$ is a tuple of length $\lg\left(y\right)$.
\end{enumerate}
Then we define $\alt\left(\varphi,B,a,p\right)$ to be the maximal
number $n$ such that there is a realization $\left\langle a_{i}\left|\, i<n\right.\right\rangle \models p^{\left(n\right)}|_{B}$,
such that $\varphi\left(a_{i},a\right)$ alternates for $i<n$, i.e.
such that $\varphi\left(a_{i},a\right)\Leftrightarrow\neg\varphi\left(a_{i+1},a\right)$
for $i<n-1.$
\end{defn}
Note that $\alt\left(\varphi,B,a,p\right)$ exists by Claim \ref{cla:infinite is bounded}
(4). Observe that $\alt\left(\varphi,B,a,p\right)\geq\alt\left(\varphi,B',a,p\right)$
when $B'\supseteq B\supseteq A$, but not necessarily the other way.
Sometimes there is equality:
\begin{lem}
\label{lem:altPerserves}Suppose $p$ is a global $A$-invariant and
$A$-dependent type, $a$ some tuple and $I$ an indiscernible sequence
over $Aa$.

Then: for every infinite subset $I'\subseteq I$ and for any formula
$\varphi\left(x,y\right)$ over $A$, $\alt\left(\varphi,IA,a,p\right)=\alt\left(\varphi,I'A,a,p\right)$.\end{lem}
\begin{proof}
Obviously, $\alt\left(\varphi,IA,a,p\right)\leq\alt\left(\varphi,I'A,a,p\right)$. 

Conversely, suppose we have some $n$ such that $\bar{a}=\left\langle a_{i}\left|\, i<n\right.\right\rangle \models p^{\left(n\right)}|_{I'A}$
alternates as in the definition. Let $\bar{x}=\left(x_{0},\ldots,x_{n-1}\right)$.
We want to show that the type 
\[
p^{\left(n\right)}\left(\bar{x}\right)|_{IA}\cup\left\{ \varphi\left(x_{i},a\right)^{\left(\mbox{if }\varphi\left(a_{i},a\right)\right)}\left|\, i<n\right.\right\} 
\]
 is consistent.

Take any finite subset and write it as $\psi\left(\bar{x},b,c\right)\wedge\xi\left(\bar{x},a\right)$
where $b\subseteq I$, $c\subseteq A$. As $I'$ is infinite, and
$I$ is indiscernible over $Aa$ we can find $b'\in I'$ such that
$b'\equiv_{Aa}b$ so $\C\models\exists\bar{x}\psi\left(\bar{x},b',c\right)\wedge\xi\left(\bar{x},a\right)$
iff $\C\models\exists\bar{x}\psi\left(\bar{x},b,c\right)\wedge\xi\left(\bar{x},a\right)$.
Now, $\psi\left(\bar{x},b,c\right)\in p^{\left(n\right)}$, and $p^{\left(n\right)}$
is $A$ invariant, hence $\psi\left(\bar{x},b',c\right)\in p^{\left(n\right)}$,
and since $\bar{a}$ satisfies $\psi\left(\bar{x},b',c\right)\wedge\xi\left(\bar{x},a\right)$,
we are done. 
\end{proof}
We will deduce Main Theorem \ref{MainThmDep} from the following theorem:
\begin{thm}
\label{thm:MainLemma}Suppose $p\left(x\right)$ is a dependent type
over $C$ with $\dpr\left(p\right)\geq\kappa$. Assume this is witnessed
by $c\models p$ and $\left\{ I_{i}\left|\, i<\kappa\right.\right\} $
where $I_{i}$ has order type $\omega$ for $i<\kappa$.

Then there are
\begin{itemize}
\item $C'\supseteq C$ with $\left|C'\backslash C\right|$ finite, $c'\models p$
and $J_{0}$
\end{itemize}
such that
\begin{itemize}
\item $\left\{ J_{0}\right\} \cup\left\{ I_{i}\left|\,0<i<\kappa\right.\right\} $
is mutually indiscernible over $C'$; $c'\equiv_{C\cup\left\{ I_{i}\left|\,0<i<\kappa\right.\right\} }c$
; $J_{0}$ is not indiscernible over $C'c'$ and 
\item all the tuples in $J_{0}$ satisfy $p$. 
\end{itemize}
\end{thm}
\begin{proof}
Denote $I_{i}=\left\langle f_{i,j}\left|\, j<\omega\right.\right\rangle $.
By compactness, we can find $f_{i,j}$ for $j\in\Zz$ and $i<\kappa$
such that, letting $I_{i}'=\left\langle f_{i,j}\left|\, j\in\mathbb{Z},j<0\right.\right\rangle $,
$\left\{ I_{i}'\concat I_{i}\left|\, i<\kappa\right.\right\} $ is
mutually indiscernible over $C$. Let $U$ be a non-principal ultrafilter
on $\omega$. For $i<\kappa$, let $p_{i}$ be global coheir over
$I_{i}'$ defined by: 
\begin{itemize}
\item for a formula $\psi\left(z,y\right)$ and a tuple $a\in\C$, $\psi\left(z,a\right)\in p_{i}$
iff $\left\{ n<\omega\left|\,\models\psi\left(f_{i,-n},a\right)\right.\right\} \in U$. 
\end{itemize}
So each $p_{i}$ is invariant over $\bigcup_{i<\kappa}I_{i}'$ and
we can consider the type $\left(\bigotimes_{0<i<\kappa}p_{i}^{\left(\omega\right)}\right)^{\left(\omega\right)}$,
and find a model $M\supseteq C\cup\bigcup_{i<\kappa}\left(I_{i}'\concat I_{i}\right)$
such that this type is an heir over $M$ (using Claim \ref{cla:Making heir}). 

Let $\left\langle K_{i}\left|\, i<\kappa\right.\right\rangle \models\bigotimes_{i<\kappa}p_{i}^{\left(\omega\right)}|_{M}$,
then:
\begin{itemize}
\item each $K_{i}$ is a Morley sequence of $p_{i}$ over $M$, 
\item since $\left\{ I_{i}'\concat I_{i}\left|\, i<\kappa\right.\right\} $
is mutually indiscernible over $C$, and $p_{i}$ is finitely satisfiable
in $I'_{i}$, $\left\langle K_{i}\left|\, i<\kappa\right.\right\rangle \equiv_{C}\left\langle I_{i}\left|\, i<\kappa\right.\right\rangle $
(this follows from the fact that the order type of $I_{i}'$ is $\omega^{*}$
--- $\omega$ in reverse), and
\item by Remark \ref{rem:Morley Seq mut. indis}, $\left\{ K_{i}\left|\, i<\kappa\right.\right\} $
is mutually indiscernible over $M$.
\end{itemize}
By the second bullet, there is an automorphism of $\C$ that fixes
$C$ (but may move $M$ and $p_{i}$) and maps $\left\langle K_{i}\left|\, i<\kappa\right.\right\rangle $
to $\left\langle I_{i}\left|\, i<\kappa\right.\right\rangle $. By
applying it we we may assume that $\left\langle K_{i}\left|\, i<\kappa\right.\right\rangle =\left\langle I_{i}\left|\, i<\kappa\right.\right\rangle $. 

Let $\mu=\left|T\right|^{+}$. Let $J=\left\langle d_{i}\left|\, i<\mu\right.\right\rangle $
be a Morley sequence of $\left(\bigotimes_{0<i<\kappa}p_{i}^{\left(\omega\right)}\right)$
over $MI_{0}$ so that $d_{0}$ is the infinite tuple $\left\langle I_{i}\left|\,0<i<\kappa\right.\right\rangle $.
Note that $I_{0}$ is indiscernible over $JM$, $J$ is indiscernible
over $I_{0}M$ and $\left\{ I_{i}\left|\, i<\kappa\right.\right\} $
is mutually indiscernible over $M\cup\left\{ d_{i}\left|\,0<i<\mu\right.\right\} $
(by Remark \ref{rem:Morley Seq mut. indis}).

Now, $I_{0}$ is not indiscernible over $Cc$. So there are increasing
tuples $a_{0}$ and $a_{1}$ from $I_{0}$ of the same length and
a formula $\varphi\left(x,y\right)$ over $C$ such that $\neg\varphi\left(c,a_{0}\right)\land\varphi\left(c,a_{1}\right)$
holds. By indiscernibility, there is an automorphism $\sigma$ of
$\C$ that fixes $JM$ and takes $a_{0}$ to $a_{1}$. Let $c_{0}=c$
and $c_{1}=\sigma\left(c_{0}\right)$. Then $\varphi\left(c_{0},a_{1}\right)\land\neg\varphi\left(c_{1},a_{1}\right)$
holds.

By Proposition \ref{prop:(shrinking-of-indiscernibles)}, for some
$\alpha<\mu$, the sequence $J'=\left\langle d_{i}\left|\,\alpha<i<\mu\right.\right\rangle $
is indiscernible over $Mc_{0}c_{1}$. By Lemma \ref{lem:IndImpliesNonForking},
$c_{0}c_{1}\ind_{M}^{f}J'$. 

Let $r\left(x\right)$ be a global non-forking (over $M$) extension
of $\tp\left(c_{0}/MJ'\right)$ ($=\tp\left(c_{1}/MJ'\right)$). Since
$\tp\left(c_{0}/C\right)$ is dependent, $r\left(x\right)$ is $M$-dependent
and $M$-invariant (by Proposition \ref{prop:non-forking}). Let $n=\alt\left(\varphi,MJ,a_{1},r\right)$,
and let $\left\langle e_{i}\left|\, i<\omega\right.\right\rangle $
be a Morley sequence of $r$ over $MJ$ that witnesses this, i.e.
such that $\varphi\left(e_{i},a_{1}\right)$ alternates for $i<n$.
Let $\left\langle e_{i}\left|\,\omega\leq i<\omega+\omega\right.\right\rangle $
be a Morley sequence of $r$ over $MJc_{0}c_{1}e_{<\omega}$. 

So:
\begin{itemize}
\item $I_{0}'=\left\langle e_{i}\left|\, i<\omega+\omega\right.\right\rangle $
is an $MJ$-indiscernible sequence, and moreover $\left\{ I_{i}\left|\,0<i<\kappa\right.\right\} \cup\left\{ I_{0}'\right\} $
is a set of $MJ'$-mutually indiscernible sequences.
\end{itemize}
Now, both $\left\langle c_{0},e_{\omega},e_{\omega+1},\ldots\right\rangle $,
$\left\langle c_{1},e_{\omega},e_{\omega+1},\ldots\right\rangle $
are Morley sequences of $r$ over $MJ'$. But if in addition $\left\langle c_{0},e_{0},e_{1},\ldots\right\rangle $
and $\left\langle c_{1},e_{0},e_{1},\ldots\right\rangle $ are also
Morley sequences of $r$ over $MJ'$, then since one of $c_{0}$,
$c_{1}$, adds an alternation of the truth value of $\varphi\left(x,a_{1}\right)$,
this is a contradiction to the choice of $e_{i}$ and to Lemma \ref{lem:altPerserves}
(which we can use because $J$ is indiscernible over $a_{1}$, and
$J'$ is infinite). Let $c'\in\left\{ c_{0},c_{1}\right\} $ be such
that $\left\langle c',e_{0},e_{1},\ldots\right\rangle $ is not a
Morley sequence of $r$ over $MJ'$. Note that $c'\equiv_{MJ}c$ and
that $MJ$ contains $C\cup\left\{ I_{i}\left|\,0<i<\kappa\right.\right\} $.

So, $\left\langle c',e_{0},\ldots\right\rangle \not\equiv_{MJ'}\left\langle c',e_{\omega},\ldots\right\rangle $
and hence the sequence $I_{0}'$ is not indiscernible over $c'MJ'$.
Let $J_{0}$ be some infinite subset of $I_{0}'$ of order type $\omega$
that witnesses this, and let $C'\supseteq C$ be such that $\left|C'\backslash C\right|$
is finite, and $C'\subseteq MJ'$ so that $J_{0}$ is not indiscernible
over $C'c'$. 

It is now easy to check that all conditions are satisfied.
\end{proof}
Now let us conclude:
\begin{proof}
(of Main Theorem \ref{MainThmDep}) Suppose $p$ is a dependent type
over $A$ with $\dpr\left(p\right)\geq\kappa$. Consider the family
$\mathcal{F}$ of triples $\left(s,c,J,A'\right)$ such that
\begin{itemize}
\item $c\models p$, $s\subseteq\kappa$, $J=\left\langle I_{i}\left|\, i<\kappa\right.\right\rangle $;
$A\subseteq A'$; $J$ is a sequence of $A'$-mutually indiscernible
sequences such that for each $i<\kappa$, $I_{i}$ is not indiscernible
over $A'c$; all tuples in $I_{i}$ for $i\in s$ realize $p$.
\end{itemize}
By assumption, $\mathcal{F}$ is not empty. Define the following order
relation on these triples:
\begin{itemize}
\item $\left(s,c,J,A'\right)\leq\left(s',c',J',A''\right)$ iff ($s\subseteq s'$,
$A'\subseteq A''$, if $i\in s\cup\left(\kappa\backslash s'\right)$
then $I_{i}=I_{i}'$ and $c'\equiv_{A'\cup\left\{ I_{i}\left|\, i\in s\cup\left(\kappa\backslash s'\right)\right.\right\} }c$).
\end{itemize}
It is easy to see that by compactness $\mathcal{F}$ satisfies the
conditions of Zorn's Lemma, so it has a maximal member $\left(s_{0},c_{0},J_{0},A'_{0}\right)$.
By Theorem \ref{thm:MainLemma}, $s_{0}=\kappa$ and we are done. 
\end{proof}
\bibliographystyle{alpha}
\bibliography{common}

\end{document}